\documentclass{amsart}
\usepackage[utf8]{inputenc}
\usepackage{combelow}
\usepackage[margin=1.4 in]{geometry}
\usepackage{amssymb,bm,amsfonts, stmaryrd, amsmath, amsthm, enumitem, mathtools,comment}
\usepackage[all]{xy}
\makeatletter
\@namedef{subjclassname@2020}{\textup{2020} Mathematics Subject Classification}
\makeatother
\usepackage[colorlinks]{hyperref}
\hypersetup{
 colorlinks=true,
 linkcolor=blue,
 filecolor=magenta, 
 urlcolor=cyan,
}
\usepackage{wrapfig,tikz, tikz-cd}
\usetikzlibrary{arrows}
\usepackage[capitalise]{cleveref}
\crefformat{equation}{(#2#1#3)}
\crefrangeformat{equation}{(#3#1#4--#5#2#6)}
\crefformat{enumi}{(#2#1#3)}
\crefrangeformat{enumi}{(#3#1#4--#5#2#6)}
\newtheorem{theorem}{Theorem}[section]
\newtheorem{lemma}[theorem]{Lemma}

\newtheorem{corollary}[theorem]{Corollary}
\newcounter{intro}
\newtheorem{questionx}{Question}
\newtheorem{introthm}[intro]{Theorem}

\theoremstyle{definition}

\newtheorem{example}[theorem]{Example}
\newtheorem{remark}[theorem]{Remark}

\newtheorem{chunk}[theorem]{}
\newtheorem{Notation}[theorem]{Notation}

\newtheorem{thm}{Theorem}[subsection]

\theoremstyle{definition}
\newtheorem{ex}[thm]{Example}

\newtheorem{ch}[thm]{}

\newtheorem{defn}[thm]{Definition}

\newcommand{\Kos}{\text{Kos}}

\newcommand{\les}{\leqslant}

\newcommand{\Spec}{{\operatorname{Spec}}}

\newcommand{\spec}{{\operatorname{Spec}^*}}

\newcommand{\Hom}{{\operatorname{Hom}}}

\newcommand{\HH}{\operatorname{HH}}
\newcommand{\Ext}{{\operatorname{Ext}}}
\newcommand{\Tor}{{\operatorname{Tor}}}
\newcommand{\RHom}{\operatorname{\mathsf{RHom}}}
\newcommand{\del}{\partial}
\renewcommand{\H}{\operatorname{H}}
\renewcommand{\S}{\mathcal{S}}

\newcommand{\supp}{{\operatorname{Supp}}}
\newcommand{\xra}{\xrightarrow}

\newcommand{\V}{{\rm{V}}}

\newcommand{\cV}{{\mathcal{V}}}

\newcommand{\ann}{{\operatorname{ann}}}

\newcommand{\lotimes}{\otimes^{\sf L}}

\newcommand{\m}{\mathfrak{m}}
\newcommand{\p}{\mathfrak{p}}
\newcommand{\q}{\mathfrak{q}}
\newcommand{\A}{\mathcal{A}}

\newcommand{\Z}{\mathbb{Z}}

\newcommand{\Id}{\text{Id}}

\title[Cohomological support varieties under local homomorphisms]{Cohomological support varieties under \\local homomorphisms}
\author[R.~Watson]{Ryan Watson}
\address{Department of Mathematics,
University of Nebraska, Lincoln, NE 68588, U.S.A.}
\email{rwatson9@huskers.unl.edu}

\subjclass[2020]{}

\begin{document}

\begin{abstract}
Given a bounded complex of finitely generated modules $M$ over a commutative noetherian local ring $R$, one assigns to it a variety, $\cV_R(M)$, called the cohomological support variety of $M$ over $R$. The variety $\cV_R(M)$ holds important homological information about the complex and the ring.  In this paper, we study the behavior of cohomological support varieties under restriction of scalars along local maps. In the case where the rings involved are complete intersections and the map is a surjective complete intersection, this recovers a theorem of Bergh and Jorgensen. Additionally, we show that if $R\to S$ is a local map of finite flat dimension, then the dimension of $\cV_R(R)$ is less than or equal to that of $\cV_S(S)$. This allows us to recover Avramov's result that the complete intersection property is preserved under localization.
\end{abstract}

\maketitle

\vspace{-.2in}
\section*{Introduction}\label{s_intro}

Inspired by geometric approaches to study modular representation theory \cite{Quillen:1971}, Avramov developed the theory of cohomological support varieties for local complete intersection rings in the 1980s \cite{avramov:1989}. Through the work in \cite{ Jorgensen:2002, Burke/Walker:2015, Avramov/Iyengar:2018, Pollitz:2019, Pollitz:2021} this theory has been extended to now include all noetherian local rings. Given a bounded complex of finitely generated modules, $M$, over a noetherian local ring $R$, one assigns a conical affine variety, $\cV_R(M)$, called the cohomological support variety of $M$ over $R$. The variety $\cV_R(M)$ holds important homological information about both $M$ and $R$. For example, $R$ is a complete intersection if and only if $\cV_R(M) = \{0\}$ for some finitely generated $R$-module $M$ \cite{Pollitz:2019}. This theory has been a valuable tool in commutative algebra and has led to many other remarkable results \cite{Avramov/Buchweitz:2000b, Avramov/Iyengar:2007, Bergh, Stevenson:2014a, Bergh/Joregensen:2015, Dao/Sanders:2017, Briggs/Grifo/Pollitz:2022, Liu/Pollitz:2025} (among others).

Throughout, $(R,\m_R,k)$ will denote a noetherian local ring with algebraically closed residue field $k$. Let $(Q,\m_Q,k)$ be a regular local ring such that the $\m_R$-adic completion of $R$, denoted by $\widehat{R}$, is isomorphic to $Q/I$ where $I$ is minimally generated by $f_1,\dotsc, f_n \in \m_Q^2$. Such a $Q$ exists by Cohen's structure theorem. Let $M$ be a bounded complex of finitely generated $R$-modules. The cohomological support variety of $M$ is a conical affine variety in $n$-dimensional affine space,
\[
\cV_R(M) \subseteq \mathbb A^n_k \cong I/\m_Q I.
\]

In this paper, we will often consider objects in the derived category of $R$, denoted $D(R)$ (see \ref{c_derived-cat}). We write $D^f(R)$ for the full subcategory of $D(R)$ consisting of objects with bounded and finitely generated homology. If $M$ and $N$ are two bounded complexes of finitely generated $R$-modules that are isomorphic in $D^f(R)$, then $\cV_R(M) = \cV_R(N)$. Thus, cohomological support varieties give a way to assign a conical affine variety to each object in $D^f(R)$. Cohomological support varieties have been used to classify complete intersections via the structure of $D^f(R)$ \cite{Pollitz:2019} as well as classify the thick subcategories of $D^f(R)$ when $R$ is a complete intersection \cite{Stevenson:2014a} (in the language of triangulated categories) and the thick subcategories of $D^f(E)$ where $E$ is the Koszul complex on $f_1, \dotsc, f_n$ \cite{Liu/Pollitz:2025}.\\

The work in this paper has been motivated by the following question:

\begin{questionx}\label{q_functoriality}
    Given noetherian local rings $R$ and $S$, and a functor 
    \[
    F\!:D^f(S) \to \! D^f(R),
    \]
    is there a relation between $\cV_R(F(M))$ and $\cV_S(M)$ for all $M$ in $D^f(S)$?
\end{questionx}

While there may not be much hope in providing a satisfying answer for an arbitrary functor $F$, there is reason to investigate this question for specific functors. The cohomological support variety of a complex $M$ contains homological information about the complex, and if one can understand how these varieties change under certain functors, one can understand how certain homological information is transferred through said functors. This question has already been examined in papers such as \cite{Bergh/Joregensen:2015}, \cite{Iyengar/Pollitz/Sanders:2022}, and \cite{Ballard/Iyengar/Lank/Mukhopadhyay/Pollitz:2025}. In the first paper, Bergh and Jorgensen consider this question when $R$ and $S$ are complete, local, complete intersections and the functor is restriction of scalars along a surjective ring map $R \to \!S$. The first main result of this paper generalizes a result of theirs to restriction of scalars along any local map of noetherian local rings. 

\begin{introthm}\label{theorem A} \looseness -1 
Let $R$ and $S$ be local rings with algebraically closed residue fields. Assume $\widehat R \cong Q/I$ and $\widehat S \cong Q'/J$ are minimal Cohen presentations. Let $\varphi \! : R \to S$ be a local homomorphism, and $\tilde \varphi\!:\!Q \to Q'$ a lift of the map $\varphi$. This induces a map $\cV_{\tilde \varphi}\!:\!I/\m_Q I \! \to \! J /\m_{Q'} J$. Then for any $M$ in $D^f(S)$ that is an element of $D^f(R)$ under restriction of scalars, 
\[
\cV_R(M) =  \cV_{\tilde \varphi}^{-1}(\cV_S(M)).
\] 
\end{introthm}

In modular representation theory, given a $kG$-module $M$ where $G$ is a finite group and $k$ is a field of characteristic $p$ such that $p$ divides the order of $G$, there is a notion of support of $M$, denoted by $\cV_G(M)$ \cite{Benson/Carlson/Rickard:1996}. In this context, there is a result analogous to \Cref{theorem A} for this notion of support when looking at the restriction functor induced by an inclusion of subgroups \cite[11.2]{Benson/Iyengar/Krause:2008}. One should note that although the statements of \Cref{theorem A} and \textit{loc.cit.\!} are similar, the proofs and machinery used are very different.

The second main result of this paper concerns local maps of finite flat dimension and dimensions of support varieties. In \cite{Pollitz:2019}, Pollitz shows that a ring is a complete intersection if and only if $\cV_R(R) = \{0\}$. Thus, the dimension of $\cV_R(R)$ is in some sense a measure of how bad the singularity of the ring is. The following theorem shows that the dimension of the cohomological support variety of a ring can only get larger along a local map of finite flat dimension.

\begin{introthm}\label{theorem B}
Let $\varphi\!:\!R \to S$ be a local map of finite flat dimension. Then,
\[
\dim \cV_R(R) \les \dim \cV_S(S).
\]
\end{introthm}

This falls in line with other theorems and conjectures that suggest a singularity can only get worse along maps of finite flat dimension. For instance, if $R\to S$ is a local map of finite flat dimension, the Cohen-Macaulay defect and the complete intersection defect of $R$ can only be smaller than those of $S$ \cite{Avramov/Foxby/Herzog:1994, Avramov/Foxby/Halperin:1985}. As a corollary of \Cref{theorem B}, we get $\dim \cV_{R_\p}(R_\p) \les \dim \cV_R(R)$ for all prime ideals $\p \in \Spec (R)$, which recovers Avramov's result that the complete intersection property localizes \cite{Avramov:1977}.

The structure of the paper is as follows. In \Cref{s_prelim}, we recall definitions and useful facts that will be used throughout. In \Cref{s_ext-scal-weak-reg}, we look at how cohomological support varieties behave when extending scalars along weakly regular maps with a key application being residual algebraic closures. In \Cref{s_Restriction of Scalars} we state and prove \Cref{theorem A} and give some examples and corollaries. In \Cref{s_ffd}, we state and prove \Cref{theorem B} and consider some corollaries and applications of said theorem.

\section*{Acknowledgements}
The author would like to thank Elo\'isa Grifo for the many hours she invested in advising the author and for discussing ideas that eventually turned into this paper. Special thanks is also owed to Josh Pollitz who had suggested \Cref{theorem B} and \Cref{co_localization-ci} during a research visit I had to Syracuse University as well as to Ben Briggs whose conversations led to finishing the proof of \Cref{t_extensions-scalars-weak-reg}. I would also like to thank Mark Walker in addition to the previous three for numerous helpful comments on previous drafts of this paper. Lastly, I would like to thank Srikanth Iyengar for pointing out the parallels between \Cref{theorem A} and the analogous result in modular representation theory.

This material is based upon work supported by the National Science Foundation under Grant No. DMS-1928930 and by the Alfred P. Sloan Foundation under grant G-2021-16778, while the author was in residence at the Simons Laufer Mathematical Sciences Institute (formerly MSRI) in Berkeley, California, during the Spring 2024 semester. The author was also partially funded by NSF grants DMS-2236983 and DMS-2342256.

\section{Preliminaries and Basic Results}\label{s_prelim}

\subsection{Fixed notation for complexes}\label{notation}
Unless otherwise stated, all rings will be assumed to be commutative noetherian local rings.

Let $R$ be a ring.
An $R$-module will be viewed as a complex concentrated in degree $0$. If $F$ is a complex and $a$ is an element in $F_n$, then $\lvert a\rvert = n$. We will generally only consider elements of a single degree in a complex.

If $F$ and $G$ are both complexes, then $\Hom_R(F,G)$ is the complex where the modules in each degree and differentials are given by
\[
\Hom_R(F,G)_d = \{\beta \!: \! F \to G \mid \beta \text{ is an $R$-linear map of degree $d$}\}
\]
\[
\del^{\Hom(F,G)}(\beta) = \del^G\circ \beta - (-1)^{\lvert \beta \rvert} \beta \circ \del^F. 
\]
This makes $\Hom_R(F,G)$ into a complex of $R$-modules whose underlying graded module structure is given by $\bigoplus _{d \in \Z} \Hom_R(F,G)_d$. Morphisms of complexes are cycles of degree $0$ in $\Hom_R(F,G)$. 

If $F$ and $G$ are both complexes, then $F \otimes_R G$ is the complex where the modules in each degree and differentials are given by
\[
(F \otimes G)_d =\bigoplus_{i+j=d} F_i \otimes F_j
\]
\[
\del^{F \otimes G} (a \otimes b) = \partial^F(a)\otimes b + (-1)^{\lvert a \rvert} a \otimes \partial^G(b).
\]

If $F$ and $G$  are $R$-complexes and $f \colon F\xra{}G$ is a morphism of complexes, then the induced map on homology is denoted by $H(f) \! :\!H(F) \xra{}H(G)$. To denote that a map is a quasiisomorphism or that two complexes are quasiisomorphic, we write $F \xra{\simeq} G$ or $F\simeq G$ respectively. 

\subsection{Differential Graded Algebras and Modules}\label{s_dga}
We make use of the theory of differential graded algebras and modules.
A good reference for such topics is \cite{Avramov:2010}. We will also work in the derived category of dg modules over a dg algebra. For more information about derived categories, the author refers the  reader to \cite{Avramov/Buchweitz/Iyengar/Miller:2010, Krause:2021, Christensen/Foxby/Holm:2024}. In this paper all dg algebras are assumed to be strictly commutative dg algebras and will be concentrated in nonnegative degrees, even if not explicitly stated. That is, for any dg algebra $A$, 
\begin{itemize}
\item  $a_1a_2 = (-1)^{\lvert a_1 \rvert \lvert a_2 \rvert}a_2a_1$ for all $a_1,a_2 \in A$.
\item $a^2 =0$ for all $a\in A$ such that $\lvert a \rvert$ is odd.
\item $A_n =0$ for all $n<0$.
\end{itemize}

\begin{ch}\label{c_derived-cat}
Let $A$ be a dg algebra. We write $A^\natural$ for the graded algebra where we forget the differential of $A$. For a dg module $M$ over $A$, we write $M^\natural$ for the underlying graded module over $A^\natural$ where we forget about the differential on $M$.
\end{ch}

\begin{ch}
Let $A$ be a dg algebra. We write $D(A)$ for its derived cateogry. That is, objects are dg $A$-modules and the morphisms can be obtained from the morphisms of dg $A$-modules, but we allow formal inverses to quasiisomorphisms. 
We write $D^f(A)$ for the full subcategory of $D(A)$ consisting of those objects $M$ such that $H(M)$ is finitely generated over $H(A)$. In this paper, all dg algebras will be bounded, so objects of $D^f(A)$ will have bounded and degree-wise finitely generated homology.
\end{ch}

\begin{ch}
Given a morphism $\varphi \colon A \to B$ between two dg algebras, there is a functor
\[
\varphi_*:D(B) \to \!D(A)
\]
which is restiction of scalars along $\varphi$. For a dg $B$-module $M$, by abuse of notation, we will write $\varphi_*(M) = M$. Note that since we are assuming dg algebras are nonnegatively graded, there is always a map of dg algebras $A \to \!H_0(A)$ for any dg algebra $A$. Therefore, any complex over $H_0(A)$ inherits a dg $A$-module structure.
\end{ch}

\begin{defn}\label{d_restrict-fullsubcat}
Let $A \to B$ be a morphism of dg algebras. Let $D^f(B/A)$ denote the full subcategory of $D^f(B)$ consisting of dg $B$-modules that restrict to elements in $D^f(A)$. That is, objects of $D^f(B/A)$ are the objects of $D^f(B)$ that land in $D^f(A)$ under restriction of scalars. Note that this notation is also used in \cite{Briggs/McCormick/Pollitz:2022}, but has a slightly different meaning.
\end{defn}

The rest of this section concerns semifree dg modules. Much of it comes from the appendix in \cite{Avramov/Iyengar/Nasseh/SatherWagstaff:2019} as well as \cite{Felix/Halperin/Thomas:2001} and \cite{Avramov:2010}. We refer the reader there for more details.

\begin{ch}
Let $A$ be a dg $R$-algebra and $F$ a dg $A$-module. A \textit{semibasis} of $F$ is a well-ordered subset $\{\mathbf{f}\} \subseteq F$ which is a basis of $F^\natural$ over $A^\natural$ and for any $f \in \{\mathbf f\}$ we have that 
\[
\partial(f) \in \sum_{e<f} Ae.
\]

A dg module $F$ that admits a (finite) semibasis is said to be \textit{(finite) semifree}.
This definition is equivalent to having a filtration of dg $A$-modules
\[
0 = F(-1) \subseteq F(0) \subseteq F(1) \subseteq \dotsm \subseteq F
\]
such that $\bigcup_{n \in\mathbb N} F(n) = F$ and the dg $A$-module $F(n)/F(n-1)$ is isomorphic to a direct sum of shifts of $A$.
\end{ch}

\begin{ch}\label{c_tensor-qis}
A crucial property of semifree modules is that tensoring with them preserves quasiisomorphisms \cite[6.7, ii)]{Felix/Halperin/Thomas:2001}.
\end{ch}

\begin{ch}\label{c_semifree-existence}
A \textit{semifree resolution} of a dg $A$-module $M$ is a quasiisomorphism of dg $A$-modules $\varepsilon \!: F \! \xra{\simeq} \!M$ such that $F$ is semifree. An important fact is that semifree resolutions always exist \cite[6.6]{Felix/Halperin/Thomas:2001}.
\end{ch}

\begin{ch}
Given a dg algebra $A$ and a dg $A$-module $M$, there are two exact functors 
\[
- \lotimes_A M : D(A) \to \!D(A) \ \ \ \text{and} \ \ \ \RHom(M,-): D(A) \to \!D(A) 
\]
and by taking homology, we define
\[
\Tor^A(-,M) := \H(-\lotimes_AM) \ \ \ \text{and} \ \ \   \Ext_A(M,-):= \H(\RHom_A(M,-)).
\]
If $F \xra{\simeq} M$ is a semifree resolution, then for any dg $A$-module $N$,
\[
\Tor^A(N,M) \cong \H(N\otimes_AF) \ \ \ \text{and} \ \ \ \Ext_A(M,N) \cong \H(\Hom_A(F,N)).
\]
\end{ch}

\subsection{Regular factorizations}\label{s_reg-fac}
In this subsection, we recall the definition and basic facts about regular factorizations. All statements and definitions can be found in \cite{Avramov/Foxby/Herzog:1994}. A local homomorphism $(R,\m) \to (T, \mathfrak n)$ is said to be \textit{weakly regular} if it is flat and the closed fiber $T/\m T$ is regular.
Given a local homomorphism $\psi:R \! \to \!S$ a \textit{regular factorization} of $\psi$ is a factorization of the form
\[
\begin{tikzcd}
&T\ar[dr,two heads, "\psi'"] & \\
R \ar[ur,"\dot \psi"] \ar[rr, "\psi"] && S,
\end{tikzcd}
\]
where $\dot \psi$ is weakly regular and $\psi'$ is surjective. When $T$ is complete, this is called a \textit{Cohen factorization}. A Cohen factorization of $\psi$ exists if and only if $S$ is complete \cite[1.1]{Avramov/Foxby/Herzog:1994}.

\begin{ex}
An example of a weakly regular morphism is $R \to \widehat R$ where $\widehat R$ is the completion of $R$ at its maximal ideal. Indeed, $R \to \widehat R$ is flat and the closed fiber is $\widehat R / \m_R \widehat R \cong \widehat R/\m_{\widehat R}$ which is a field and thus regular. 
\end{ex}

\begin{ch}
Our interest in regular factorizations in this paper stems from the fact that to prove \Cref{theorem A}, one only has to consider weakly regular maps and surjective maps separately. 
\end{ch}

\begin{ch}\label{c_lift-weakly-reg}
Given a surjective local homomorphism $\pi\!:\!Q \! \to \!R$ and a weakly regular homomorphism $\psi\!:\!R \!\to \!S$, such that $S$ is complete, one may construct maps $\tilde \psi \!:\!Q \to Q'$ and $\pi'\!:\!Q' \to S$ such that following commutes
\[
\begin{tikzcd}
Q \ar[rr,"\tilde \psi"] \ar[dd,"\pi", swap] && Q' \ar[dd,"\pi'"] \\ \\
R \ar[rr,"\psi"] && S
\end{tikzcd}
\]
where $Q'$ is complete, $\tilde \psi$ is weakly regular, $\pi'$ is surjective and $S \cong Q' \otimes_Q R$. 
\end{ch}

\begin{ch}
A consequence of the definitions of regular local rings and weakly regular maps yields the following fact. If $Q \!\to \!Q'$ is weakly regular and $Q$ is a regular local ring, then $Q'$ is also a regular local ring. Thus \ref{c_lift-weakly-reg} says that if $\pi\!:\!Q\! \to \!R$ is a Cohen presentation and $\psi\!:\!R \! \to \!S$ is a weakly regular map of complete local rings, then one can find a Cohen presentation $\pi'\!:\!Q' \!\to \!S$ such that $\psi\!:\! R \! \to \! S$ lifts to a weakly regular map $\tilde \psi\!:\!Q \! \to \!Q'$. Moreover, if $\pi\!:\!Q \to R$ is a minimal Cohen presentation, then $\pi'\!:\!Q' \to S$ is as well \cite[1.5 and 1.6]{Avramov/Foxby/Herzog:1994}. 
\end{ch}

\subsection{Graded support}
Let $\S  = Q[\chi_1, \dotsc, \chi_n]$ be a graded polynomial ring over a commutative ring $Q$ with each $\chi_i$ having cohomological degree $2$. Let $\spec(\S)$ be the set of homogeneous prime ideals of $\S$. We equip $\spec \S$ with the Zariski topology. That is, closed subsets are of the form $
\mathbf V(\mathcal{I})\coloneqq\{\p\in \spec \S:\mathcal{I}\subseteq \p\}
$ 
for a homogeneous ideal $\mathcal{I}$ of $\S$. When $Q$ is a field, we can view points inside of $\spec \S$ as conical affine varieties in $n$-dimensional affine space $\mathbb A^n_Q$, with the unique homogeneous maximal ideal $(\chi_1,\dotsc,\chi_n)$ corresponding to the origin.

By \cite[2.4]{Carlson/Iyengar:2015}, the support of a finitely generated $\S$-module $X$ is given by
\[
\supp_\S X\coloneqq \{\p\in \spec \S: X_\p\not \simeq 0\}\ = \{\p \in \spec S \mid \kappa(\p)\lotimes_{\S} X \not \simeq 0\}.
\]
The support of a finitely generated graded $\S$ module $X$ is a closed subset of $\spec \S$,
\[
\supp_\S X=\mathbf V(\ann_\S X)\,.
\]

\subsection{Cohomological support varieties}\label{s_support}
In this subsection, we recall the definition of cohomological support varieties as well as some useful properties. The theory of support varieties was first introduced in commutative algebra by Avramov in \cite{avramov:1989} who defined them over local complete intersections. D. Jorgensen extended the definition of these support varieties using intermediate hypersurfaces 
to all local rings \cite{Jorgensen:2002}. In \cite{Pollitz:2021}, Pollitz recovers these two theories with the theory of cohomological support varieties over derived complete intersections. We refer the reader to \textit{loc.cit.\!} for more details.

\begin{ch}\label{c_cohomological-op}
Let $Q$ be a commutative noetherian ring, and $f_1,\dotsc, f_n \in Q$. Let $E$ be the Koszul complex, $E=\Kos^Q(f_1,\dotsc,f_n)$, and let $M$ and $N$ be dg $E$-modules. Denote the \textit{ring of cohomological operators} by $\S:=Q[\chi_1,\dotsc,\chi_n]$, where each $\chi_i$ has cohomological degree $2$. By \cite[4.2.5]{Pollitz:2021}, the Hochschild cohomology $\HH(E\vert Q)$, is given by 
\[
\HH(E\vert Q) = \H(E)[\chi_1, \dotsc, \chi_n]
\]
and so $\S$ surjects onto the graded subalgebra $\H_0(E)[\chi_1, \dotsc, \chi_n] \subseteq \HH(E\vert Q)$. Hence the left action of $\Ext_E(N,N)$ on $\Ext_E(M,N)$ via the composition product turns $\Ext_E(M,N)$ into a graded $\S$-module via the ring map $\HH(E\vert Q) \to \Ext_E(N,N)$ \cite[2.3]{Briggs/Iyengar/Letz/Pollitz:2020}. When $(Q,\m ,k)$ is a local ring, we define $\A := \S \otimes _Q k$.
\end{ch}

\begin{ch}
By \cite[5.1.3]{Pollitz:2021}, if $M$ and $N$ are objects in $D^f(E)$ such that $\Ext_Q^{\gg 0} (M,N)=0$, then $\Ext_E(M,N)$ is a finitely generated graded $\S$-module. Note this always holds when $Q$ is a regular local ring. 
\end{ch}

\begin{ch}\label{c_intersection-coho-supp}
Using the notation in \ref{c_cohomological-op}, the \textit{cohomological support variety} of a pair of objects $M$ and $N$ in $D^f(E)$ is 
\[
\V_E(M,N) := \supp_{\S}( \Ext_E(M,N) )\subseteq \spec \S.
\]
When $(Q, \m,k)$ is local, we define
\[
\cV _E(M,N) := \supp_{\A}( \Ext_E(M,N)\otimes_Q k) \subseteq \spec \A.
\]
For a single object $M \in D^f(E)$, we define its cohomological support variety to be 
\[
\V_E(M) := \V_E(M,M),
\]
and we define $\cV_E(M)$ similarly when $Q$ is local. Moreover, when $Q$ is a regular local ring, by \cite[5.3.1]{Pollitz:2021}, we have
\[
\cV _E(M) = \cV _E(M,k) = \cV _E(k,M).
\]
\end{ch}

\begin{ch}\label{c_SV-rings}
Let $R$ be a local ring and $(Q, \m, k) \! \twoheadrightarrow \!\widehat{R}$ a minimal Cohen presentation. That is, $\widehat{R} \cong Q/(f_1,\dotsc,f_n)$ where $Q$ is a complete regular local ring and $f_1, \dotsc, f_n \in \m^2$ minimally generate $(f_1,\dotsc, f_n)$. We define the \textit{cohomological support variety} of a pair of objects $M$ and $N$ in $D^f(R)$ to be 
\[
\cV _R(M,N) :=\supp_{\A}( \Ext_E(\widehat{M},\widehat{N})\otimes_Q k) \subseteq \spec \A.
\]
where $\widehat M$ and $\widehat N$ denote $M \otimes_R \widehat R$ and $N \otimes_R \widehat R$ respectively. We define the cohomological support variety of a single object $M$ in $D^f(R)$ as 
\[
\cV_R(M) := \cV_R(M,M).
\]
Thus $\cV _R(M) = \cV _E(\widehat{M})$ where $E = \Kos^Q(f_1,\dotsc,f_n)$. Note that these definitions depend on the choice of minimal Cohen presentation but by \cite[6.1.2]{Pollitz:2021}, different choices of minimal Cohen presentations give isomorphic varieties. However, the ambient space in which we view $\cV_R(M)$ in changes depending on what minimal Cohen presentation is chosen. We deal with this technical detail in \Cref{l_compatability-Cohen-pres}.
\end{ch}

\begin{ch}\label{c_intermediate-hyperplane}
By \cite[6.2.4]{Pollitz:2021}, we have the following fact. Let $(Q, \m,k)$ be a regular local ring with algebraically closed residue field. Let $E = \Kos^Q(f_1,\dotsc, f_n)$, where  $f_1,\dotsc,f_n \in \m^2$ minimally generate an ideal $I$. Then viewing $\cV_E(M)$ as an affine variety in $\mathbb A^n_k \cong I/\m I$,
\[
\cV_E(M) = \{\overline f\in I/\m I\mid \Ext_{E_f}(M,k) \text{ is unbounded}\} \cup \{0\}
\]
where $E_f=\Kos^Q(f)$ and $f$ is a lift of $\overline f$ to $Q$. Note that by \cite[6.1]{Felix/Halperin/Thomas:2001},
\[
\Ext_{E_f}(M,k) \cong \Ext_{Q/(f)}(M,k).
\]
\end{ch}

\section{Extension of Scalars Along Weakly Regular Maps}\label{s_ext-scal-weak-reg}
In this section, we study what happens to cohomological support varieties when we extend scalars along a weakly regular map. Two points of interest will be completions and residual algebraic closures, whose definition we recall further down.

We set the following notation throughout this section. Let $Q$ and $Q'$ be regular local rings, and $\tilde \varphi\!:\!Q \to Q'$ a flat local homomorphism. Let $R=Q/I$ and $R' = R \otimes_Q Q'$ and define $J = IQ'$ so that $R' = Q'/J$.  Assume there exists a flat local homomorphism $\varphi\!:\!R \to R'$ making the following diagram commute:
\[
\begin{tikzcd}
Q\ar[r, "\tilde \varphi"] \ar[d, two heads] & Q' \ar[d, two heads] \\
R\ar[r, "\varphi"] & R'
\end{tikzcd}
\]
Let $f_1,\dotsc, f_n \in Q$ generate the ideal $I$, so $\tilde \varphi(f_1),\dotsc,\tilde \varphi(f_n) \in Q'$ generate $J$. We define the following Koszul complexes:
\[
E=\Kos^Q(f_1,\dotsc,f_n)\quad \text{and} \quad
E' =\Kos^{Q'}(\tilde \varphi(f_1), \dotsc, \tilde \varphi(f_n)).
\]
Let $\S = Q[\chi_1,\dotsc, \chi_n]$ be the cohomological operators corresponding to $E$ and $\S'=Q'[\chi_1',\dotsc, \chi_n']$ be the cohomological operators corresponding to $E'$. Let $\Psi\colon \S \to \S'$ be the map of graded rings induced by $\tilde \varphi$ sending $\chi_i \to \chi_i'$.

\begin{theorem}\label{t_extensions-scalars-weak-reg}
Using the notation above, for any two objects $M$ and $N$ in $D^f(E)$ we have,
\[
\V_{E}(M,N) = \Psi^*(\V_{E'}(M\otimes_QQ',N \otimes_Q Q')),
\]
where $\Psi^* :\spec S' \to\spec \S$ is the induced map on $\spec$.
\end{theorem}

\begin{proof}
Note that by our assumptions, $E' \cong E \otimes_Q Q'$. Following the proof of \cite[2.2]{Avramov/Buchweitz:2000b} and using \cite[4.4.3]{Pollitz:2021}, the map
\begin{equation}\label{eq_1}
\Ext_{E'}(M\otimes_QQ',N \otimes_QQ') \to \Ext_{E}(M,N \otimes_QQ')
\end{equation}
is a bijection. Since $M$ and $N$ are objects of $D^f(E)$ and $Q'$ is flat over $Q$, the map 
\[
\Ext_{E}(M, N \otimes_QQ') \to \Ext_{E}(M,N)\otimes_QQ'
\]
is a bijection as well \cite[12.9.10]{Yekutieli:2019}. Moreover, these bijections are compatible with the isomorphism of graded $Q'$-algebras $\S' \cong \S \otimes_QQ'$. Indeed, adopting the notation from \cite[4.4.3]{Pollitz:2021}, let $\alpha\colon M \to M\otimes_QQ'$ be the $E$-linear map defined by $\alpha(m) = m \otimes1$ and $\beta = \Id_{N\otimes_QQ'}$. Notice that the morphism in \Cref{eq_1} is the natural map of $\Ext$ modules,
\[
\Ext_{\Phi}(\alpha,\beta) \colon \Ext_{E'}(M\otimes_Q Q', N \otimes_Q Q') \to \Ext_E(M,N\otimes_QQ'),
\]
where $\Phi \colon E \to E'$ is the dg algebra morphism defined by extending the map $Q \to Q'$ by sending the $i^\text{th}$ degree $1$ generator of $E$ to the $i^\text{th}$ degree $1$ generator of $E'$.
By \textit{loc.cit.}, the following equality holds which proves the claim,
\[
\Ext_{\Phi}(\alpha,\beta) \circ \chi'_i = \chi_i \circ \Ext_{\Phi}(\alpha, \beta).
\]
Thus, for any $\mathfrak q \in \spec S'$, setting $\p:= \mathfrak q \cap \S$, we have the following:
\[
\Ext_{E'}(M\otimes_QQ', N \otimes_Q Q') \lotimes_ {\S'} \kappa(\mathfrak q)  \simeq \Ext_{E}(M,N)\lotimes_{\S}\kappa(\p) \lotimes_{\kappa(\p)} \kappa(\mathfrak q).
\]
Since $\kappa(\p) \! \to \! \kappa(\mathfrak q)$ is faithfully flat, 
\[
\Ext_{E'}(M\otimes_QQ',N\otimes_QQ') \lotimes_{\S'}\kappa(\mathfrak q) \simeq 0  \iff  \Ext_{E}(M,N) \lotimes_{\S}\kappa(\p) \simeq 0 .
\]
Thus, \[
\mathfrak q \in \V_{E'}(M\otimes_Q Q', N \otimes_Q Q')  \iff \p \in \V_{E}(M,N).
\]
In other words,
\[
\V_{E'}(M\otimes_QQ',N\otimes_QQ') = \{\mathfrak q \in \spec S' \mid \mathfrak q \cap \S \in \V_{E}(M,N)\} = {\Psi^*}^{-1}(\V_{E}(M,N)).
\]

Since $Q\! \to \!Q'$ is a flat, local homomorphism, it is faithfully flat. Thus, $\Psi\!:\! \S \to \S'$ is also faithfully flat, so $\Psi^*$ is surjective. Hence, the claim follows from the equality above.
\end{proof}

\begin{chunk}
If $\varphi \! : \! R\to R'$ is a weakly regular map between local rings (not necessarily quotients of regular local rings), then the completion, $\widehat \varphi \! : \! \widehat R \to \widehat {R'}$, is weakly regular as well. Invoking \Cref{c_lift-weakly-reg}, we get a commutative diagram
\[
\begin{tikzcd}
Q \ar[r,"\tilde \varphi"] \ar[d,"\pi", swap, two heads] & Q' \ar[d,"\pi'", two heads] \\
\widehat R \ar[r,"\varphi"] & \widehat{R'}
\end{tikzcd}
\]
where the vertical maps give minimal Cohen presentations, the horizontal maps are weakly regular, and $\widehat{R'}\cong \widehat R \otimes_Q Q'$. Thus, we may use \Cref{t_extensions-scalars-weak-reg} on the completion of $\varphi\! : \! R \to R'$.
\end{chunk}

\begin{corollary}
Using the setup above, let $\widehat R \cong Q/I$ where $I$ is minimally generated by $f_1,\dotsc,f_n$. If $\tilde\varphi(f_1),\dotsc, \tilde \varphi(f_n)$ minimally generates $IQ'$, then for $M$ and $N$ objects in $D^f(R)$,
\[
\V_{R}(M,N) = \Psi^*(\V_{R'}(M\otimes_RR',N \otimes_R R')).
\]
\end{corollary}

\begin{proof}
The statement follows from the following chain of equalities:
\begin{align*}
\V_{R}(M,N) &:= \V_{\widehat{R}}(\widehat M, \widehat N) \\
&\cong \Psi^*(\V_{\widehat{R'}}(\widehat M \otimes_Q Q', \widehat N \otimes_Q Q')) \quad \text{by \Cref{t_extensions-scalars-weak-reg}} \\
&\cong \Psi^*(\V_{\widehat{R'}} (\widehat M \otimes_{\widehat{R}} \widehat{R'}, \widehat N \otimes_{\widehat{R}} \widehat{R'})) \quad \text{since } \widehat{R'} \cong \widehat{R} \otimes_Q Q'\\
&\cong \Psi^*(\V_{\widehat{R'}}(\widehat{M\otimes_RR'}, \widehat{N\otimes_RR'})) \quad \text{by \cite[13.2.5]{Christensen/Foxby/Holm:2024}} \\
&=: \Psi^*(\V_{R'}(M\otimes_R R', N \otimes_R R')). \qedhere
\end{align*}
\end{proof}

\begin{remark}\label{r_fiber-product}
Using the setup in \Cref{t_extensions-scalars-weak-reg}, $\V_{E'}(M\otimes_QQ',N \otimes_QQ')$ is isomorphic to the fiber product
\[
\V_{E}(M,N) \times_{\spec Q} \spec Q'.
\]
Indeed, we have the following isomorphisms:
\begin{align*}
\V_{E'}(M\otimes_Q Q', N \otimes_Q Q') &\cong \mathbf V(\ann_{\S'}\Ext_{E'}(M\otimes_QQ',N\otimes_QQ')) \\
&\cong \mathbf V(\ann_S \Ext_{E}(M,N)) \otimes_QQ') \\ 
&\cong \V_{E}(M,N) \times_{\spec Q}  \spec Q'. 
\end{align*}
The first and third isomorphisms are by definition and the second isomorphism comes from the following argument. Since $Q\to Q'$ is flat and $\Ext_E(M,N)$ is a finitely generated $\S$-module, we have an isomorphism of graded $\S'$-modules,
\[
(\ann_{\S}\Ext_E(M,N) )\otimes_QQ' \cong \ann_{\S'}(\Ext_{E}(M,N)\otimes_Q Q').
\]
By the proof of \Cref{t_extensions-scalars-weak-reg} we have an isomorphism of graded $\S'$-modules,
\[
\Ext_E(M,N) \otimes_QQ' \cong \Ext_{E'}(M\otimes_QQ',N \otimes_QQ'),
\]
which proves the statement.

Moreover, since $\cV_{E}(M,N)$ is isomorphic to $\V_{E}(M,N) \times_{\spec Q} \spec k$, the following holds
\[
\cV_{E'}(M \otimes_Q Q', N \otimes_Q Q') \cong \cV_{E}(M,N)\times_{\spec k} \spec k'.
\]
\end{remark}

\begin{remark}\label{r_dim-equality}
The previous remark implies that for any weakly regular map of local rings, $\varphi \! : \! R \to R'$, we have $\dim \cV_R(M) = \dim \cV_{R'}(M\otimes_QQ')$. This follows from noether normalization since our varieties are affine $k$-schemes of finite type.
\end{remark}

\begin{chunk}\label{c_dim-equality-cohen}
Suppose $(R,\m,k)$, even though it is not complete, is a quotient of a regular local ring, say $R \cong Q/(f_1,\dotsc, f_n)$. If $E = \Kos^Q(f_1,\dotsc, f_n)$, then for any object $M \in D^f(E)$ \[
\cV_E(M) \cong \cV_R(M).
\]
This follows from \cite[6.1.2]{Pollitz:2021}, and the fact that completions are weakly regular, along with \Cref{r_fiber-product}. Indeed, importing notation from above, let $Q' = \widehat Q$ and  $E' = E \otimes_QQ'$.
\begin{align*}
\cV_R(M) &\cong \V_{E'}(\widehat M, \widehat M) \times_{\spec Q'} \spec k \\
&=(\V_E(M,M)\times_{\spec Q} \spec Q') \times_{\spec Q'} \spec k \\
&= \V_E(M,M) \times_{\spec Q}\spec k \\
&= \cV_E(M)
\end{align*}

Note that $\cV_R(M)$ is defined in terms of the completions $\widehat R$ and $\widehat M$, but this says that if your ring is already the quotient of a regular local ring, you don't need to complete.
\end{chunk}

\begin{chunk}\label{c_residual-algebraic-closure}
A \textit{residual algebraic closure} of a local ring $(Q,\m,k)$ is a flat extension of local rings $Q \to Q'$ such that $\m Q'$ is the maximal ideal $\m'$ of $Q'$ and the induced map $k \to Q'/\m'$ is the embedding of $k$ into its algebraic closure $\overline{k} \cong Q'/\m'$. Hence, residual algebraic closures are examples of weakly regular maps. Such extension always exist \cite[App., Theor\'eme 1, Corollaire]{Bourbaki:1983}.
\end{chunk}

\begin{chunk}
In \cite{Avramov/Buchweitz:2000b} and \cite{Jorgensen:2002}, their definitions of support varieties either assume you are working over an algebraically closed field or have included taking a residual algebraic closure into the definition. However, the definition given in \cite{Pollitz:2021} (and in this paper) does not make any assumptions on the residue field and this section tells one how the support variety changes when you take a residual algebraic closure.
\end{chunk}

\section{Restriction of Scalars}\label{s_Restriction of Scalars}

In this section, we will state and provide a proof to \Cref{theorem A}. We fix the following notation. If $\varphi\!:\!R\to S$ is a local map of local rings, consider the completion $\widehat{\varphi}\!:\!\widehat R \to \widehat S$. Suppose we have Cohen presentations $Q \twoheadrightarrow\widehat R$ and $Q' \twoheadrightarrow \widehat S$, and a local map $\tilde \varphi \colon Q \to Q'$ such that the following diagram commutes: 
\[
\begin{tikzcd}
Q \ar[r,"\tilde \varphi"] \ar[d, two heads] & Q' \ar[d, two heads] \\
\widehat R\ar[r,"\widehat \varphi"] & \widehat S.
\end{tikzcd}
\]
One can always find such a diagram by taking a Cohen presentation $Q \twoheadrightarrow \widehat R$ and taking a Cohen factorization of the map $Q \to \widehat S$.
If $\widehat R\cong Q/I$ and $\widehat S \cong Q'/J$ for some ideals $I\subseteq Q$ and $J\subseteq Q'$, there is an induced map $\cV_{\tilde \varphi}\!:\! I/\m_QI \to J/\m_{Q'}J$. Note that the map $\cV_{\tilde \varphi}$ depends on the choices of Cohen presentations and the lift of the map $\widehat \varphi$. However, if one chooses different lifts of the map $\widehat \varphi\!:\! \widehat R \to \widehat S$ to Cohen presentations, say $\tilde \varphi_1\!:\!P \to P'$ and $\tilde \varphi_2\!:\!Q \to Q'$, one can relate the two induced maps $\cV_{\tilde \varphi_1}$ and $\cV_{\tilde \varphi_2}$ in a way we make precise below.

\begin{lemma}\label{l_compatability-Cohen-pres}
Let $\varphi\!:\!R\to S$ be a local homomorphism between local rings with algebraically closed residue fields. Suppose we have two commutating diagrams with vertical arrows giving minimal Cohen presentations:
\[
\begin{tikzcd}
P\ar[r, "\tilde \varphi_1"] \ar[d, two heads] & P'\ar[d, two heads] \\
\widehat R\ar[r,"\widehat \varphi"] & \widehat S
\end{tikzcd} \quad \quad \quad
\begin{tikzcd}
Q \ar[r, "\tilde \varphi_2"] \ar[d, two heads] & Q' \ar[d, two heads] \\
\widehat R \ar[r, "\widehat \varphi"] & \widehat S.
\end{tikzcd}
\]
Suppose $I_1,J_1,I_2,$ and $J_2$ are ideals such that 
\[
\widehat R \cong P/ I_1, \quad \widehat S \cong P'/J_1 \quad \text{ and } \quad \widehat R \cong Q/I_2, \quad \widehat S \cong Q'/J_2.   
\]
There exist isomorphisms of vector spaces,
\[
\varkappa\!:\!I_1/\m_PI_1 \to I_2/\m_QI_2, \quad \text{ and } \quad \varkappa'\!:\!J_1/\m_{P'}J_1 \to J_2/\m_{Q'}J_2,
\]
such that for any $M$ an object of $D^f(R)$ and $N$ an object of $D^f(S)$, $\varkappa$ maps $\cV_R(M)$ viewed inside of $I_1/\m_PI_1$ bijectively onto $\cV_R(M)$ viewed inside of $I_2/\m_QI_2$ and similarly for $\varkappa'$ with $\cV_S(N)$. Moreover, $\varkappa'\circ\cV_{\tilde \varphi_1} \circ \varkappa^{-1} = \cV_{\tilde \varphi_2}$.
\end{lemma}

\begin{proof}
By \cite[4.1.4]{Pollitz:2021}, there exist surjections of complete regular local local rings
\[
\begin{tikzcd}
&O\ar[dr, two heads] \ar[dl, two heads]& \\
P&&Q
\end{tikzcd} \quad \text{ and } \quad
\begin{tikzcd}
&O'\ar[dr, two heads] \ar[dl, two heads]&\\
P'&&Q'.
\end{tikzcd}
\]
Note that $O$ and $O'$ present the pullbacks $P\times_{\widehat{R}} Q$ and $P' \times_{\widehat S}Q'$ respectively so we may choose $O$ and $O'$ such that there exists a lift of the natural map between the pullbacks, $P\times_{\widehat R}Q \to P' \times_{\widehat S}Q'$.
Since $P,Q,P',$ and $Q'$ present minimal Cohen presentations of $R$ and $S$, by \cite[5.9.1]{Avramov/Buchweitz:2000b} we can verify the existence of the maps 
\[
\varkappa:I_1/\m_PI_1 \to I_2/\m_QI_2 \quad  \text{ and } \quad \varkappa':J_1/\m_{P'}J_1 \to J_2/\m_{Q'}J_2,
\] 
and that they are isomorphisms of vector spaces. It follows by \cite[5.3]{Avramov/Buchweitz:2000b}, that the statements about $\varkappa$ and $\varkappa'$ mapping the cohomological support varieties bijectively onto one another holds. Note that the statement in \textit{loc.cit.\!} is only for modules over a local complete intersection ring, but the same proof, \textit{mutatis mutandis}, holds for bounded complexes consisting of finitely generated modules over any noetherian local ring. Hence, the last thing we need to show is,
\[
\varkappa' \circ \cV_{\tilde \varphi_1} \circ \varkappa^{-1} = \cV_{\tilde \varphi_2}.
\]
Consider the following commutative diagram:
\[
\begin{tikzcd}
&&&& O' \arrow[ld, two heads] \arrow[rd, two heads]&\\
& O \arrow[rrru] \arrow[ld, two heads]  && P' \arrow[rd, two heads] && Q' \arrow[ld, two heads]  \\
P \arrow[rrru,] \arrow[rd, two heads] && Q \arrow[rrru, crossing over] \arrow[ld, two heads] \arrow[lu, twoheadleftarrow, crossing over] && \widehat S & \\
& \widehat R \arrow[rrru]&& &&.              
\end{tikzcd}
\]
Let $I_3 \subseteq O$ and $J_3 \subseteq O'$ be ideals such that $\widehat R \cong O/I_3$ and $\widehat S \cong O'/J_3$. By construction, the map $\varkappa\!:\!I_1/\m_pI_1 \to I_2/\m_QI_2$ factors through $I_3\cap \m_O^2 /I_3\m_O$ and similarly, $\varkappa'$ factors through $J_3 \cap \m_{O'}^2/\m_{O'}J_3$ \cite[5.9]{Avramov/Buchweitz:2000b}. Hence we have the following diagram which commutes as each map is induced by the commutative diagram above:
\[
\begin{tikzcd}                                                                       && J_3\cap \m_{O'}^2 / \m_{O'} J_3 \ar[<->, ldd] \ar[<->, rdd, shift right] &\\
& I_3\cap \m_O^2 / \m_O I_3 \ar[ru] \ar[<->,ldd]  &&\\
& J_1/\m_{P'}J_1 \ar[<->, rr, "\varkappa'"] && J_2/\m_{Q'} J_2. \\
I_1/ \m_PI_1 \ar[ru, "\cV_{\tilde \varphi_1}",swap] \ar[<->,rr, "\varkappa"] && I_2/ \m_QI_2 \ar[ru, "\cV_{\tilde \varphi_2}",swap] \ar[<->,luu, crossing over]&       
\end{tikzcd}
\]
In particular, the bottom square commutes so the following equality holds
\[
\varkappa' \circ \cV_{\tilde \varphi_1} \circ \varkappa^{-1} = \cV_{\tilde \varphi_2}. \qedhere
\]
\end{proof}

\begin{Notation}
The previous lemma says if we choose different lifts of the map $\widehat \varphi\!:\! \widehat R \to \widehat S$ to maps between minimal Cohen presentations of $\widehat R$ and $\widehat S$, say $\tilde \varphi_1\!:\!P \to P'$ and $\tilde \varphi_2\!:\! Q\to Q'$, then for any object $N$ of $D^f(S/R)$, whether we view $\cV_S(N)$ in the affine space corresponding to the presentation $P' \twoheadrightarrow \widehat S$, or that of $Q' \twoheadrightarrow \widehat S$, we have 
\[
\cV_{\tilde \varphi_1}(\cV_S(N)) \cong \cV_{\tilde \varphi_2}(\cV_S(N)).
\]
In light of this, for the rest of the paper we write $\cV_\varphi$ instead of $\cV_{\tilde \varphi}$ and won't make a distinction as to what lift of $\varphi$ is being taken.
\end{Notation}

\begin{chunk}
The proof of \Cref{theorem A} will be contingent on two key lemmas. We first show a version of \Cref{theorem A} for local surjections, and we then show a version of \Cref{theorem A} for weakly regular maps. Using \Cref{s_reg-fac} we can piece together a proof of \Cref{theorem A} from those two lemmas. The proofs of them heavily rely on \Cref{c_intermediate-hyperplane}. For the convenience of the reader we restate this fact.

Let $(Q, \m,k)$ be a regular local ring with algebraically closed residue field and $I$ an ideal of $Q$ such that $I$ is minimally generated  by elements $f_1,\dotsc, f_n \in \m^2$. Viewing $\cV_E(M)$ as an affine variety in $I/\m I \cong \mathbb A^n_k$, we have
\[
\cV_E(M) = \{\overline f\in I/\m I\mid \Ext_{Q/(f)}(M,k) \text{ is unbounded}\} \cup \{0\}.
\]
\end{chunk}

\begin{lemma}\label{l_restriction-surjections}
Let $\varphi\!: \!R \to S$ be a local surjection of local rings with algebraically closed residue field $k$. Then for any object $M$ in $D^f(S)$, we have
\[
\cV_R(M) = \cV_{\varphi}^{-1}( \cV_S(M))
\]
at the level of affine varieties.
\end{lemma}

\begin{proof}
Since the completion of a surjective local homomorphism is still surjective, we may assume that $R$ and $S$ are already complete. Let $(Q,\m_Q,k)$ be a regular local ring such that $R\cong Q/I $ is a minimal Cohen presentation. By \Cref{l_compatability-Cohen-pres}, we may consider a specific minimal Cohen presentation of $S$ which we construct in the following way.  Let $\gamma\!:\!Q \!\to S$ be the local surjection given by the composition of the map of $Q$ onto $R$ and $\varphi\!:\!R \to S$. Let $r = \text{edim}(Q)$ and $s = \text{edim}(S)$. Since $Q$ surjects onto $S$ via a local homomorphism, we can choose $s$ minimal generators of $\m_Q$ that map onto a minimal generating set of $\m_S$. Denote these $s$ elements of $Q$ by $x_1,\dotsc, x_s$. We complete this to a minimal generating set of $\m_Q$, say $x_1, \dotsc,x_s,x_{s+1},\dotsc, x_r$. There exist $\overline a_{ij} \in S$ such that
\[
\gamma(x_{s+i}) = \sum_{j=1}^s \overline a_{ij}\gamma(x_j)
\]
for all $1\les i \les r-s$. Let $a_{ij}$ be lifts of $\overline a_{ij}$ to $Q$ and define
\[
\ell_i = x_{s+i} - \sum_{j=1}^s a_{ij}x_j \ \text{ for } \ 1 \les i \les r-s, \quad \text{and} \quad Q' = \frac{Q}{(\ell_1,\dotsc,\ell_{r-s})}.
\]
Note that each $\ell_i$ is in $\mathfrak m_Q \setminus\m_Q^2$ and $\ell_1, \dotsc, \ell_{s-r}$ form a regular sequence. Thus, $Q'$ is a regular local ring with embedding dimension $s$, which is the embedding dimension of $S$. Since $\gamma(\ell_i) =0$ for all $i$, we have that $\gamma$ factors though $Q'$. Thus we have the following commutative diagram of local rings,
\[
\begin{tikzcd}
Q\ar[rr, two heads, "\tilde \varphi"] \ar[dd, two heads, swap,"\pi_R"] && Q' \ar[dd, two heads, "\pi_S"] \\ \\
R \ar[rr, two heads, "\varphi"] && S
\end{tikzcd}
\]
where the vertical maps give minimal Cohen presentations. \newline

Let $J = \ker(\pi_S)$ and $L =(\ell_1, \dotsc, \ell_{r-s})$. Let $M$ be an object of $D^f(S)$. Let $f \in Q$ be a minimal generator of $I$. That is, $\overline f \not = 0$ in $I/\m I$. First assume that $\tilde \varphi(f) \neq 0$. Since $L \cap \m_Q^2 = \mathfrak m_QL$, by \cite[2.2] {Pollitz/Sega:2025}, we have that $\Ext_{Q/(f)}(M,k)$ is unbounded if and only if $\Ext_E(M,k)$ is unbounded, where $E= \Kos^{Q/(f)}(\ell_1, \dotsc, \ell_{r-s})$. But $\ell_1,\dotsc, \ell_{r-s},f$ form a regular sequence in $Q$, so by \cite[6.1]{Felix/Halperin/Thomas:2001}, $\Ext_E(M,k) \cong \Ext_{Q'/(\tilde \varphi(f))}(M,k)$. Thus we have that 
$\Ext_{Q/(f)}(M,k)$ is unbounded if and only if $\Ext_{Q'/(\tilde \varphi(f))}(M,k)$ is unbounded. Hence,
\[
\overline{f} \in \cV_R(M) \iff \cV_{\varphi}(\overline f) \in \cV_S(M)
\]
when $f$ is a minimal generator of $I$ and $\tilde \varphi(f) \neq 0$.

Now, assume that $\tilde \varphi(f) = 0$, so $f \in L$. Since $f$ is in $\m_Q^2$, this says that $f \in \mathfrak m L \subseteq \m \ann_Q(M)$. Hence $\Ext_{Q/(f)}(M,k)$ is unbounded by \textit{loc.cit.\!} and \cite[2.2]{Pollitz/Sega:2025}. Thus, $\overline f \in \cV_R(M)$ which shows that for any nonzero $\overline f \in I/\m I$,
\[
\overline f \in \cV_R(M) \iff \cV_{\varphi}(\overline f) \in \cV_S(M).
\]

Lastly if $\overline f = 0$, then $\cV_{\varphi}(\overline f) =0$ as well, and thus the proof is complete.
\end{proof}

\begin{lemma}\label{l_restriction-weak-reg}
Let $\varphi\!:\!R \to S$ be a weakly regular morphism of local rings with algebraically closed residue fields. Let $M$ be an object of $D^f(S/R)$. Then,
\[
\cV_R(M) = \cV_{\varphi}^{-1}(\cV_S(M))
\]
at the level of affine varieties.
\end{lemma}

\begin{proof}
The completion of a weakly regular map is once again weakly regular so we may assume that $R$ and $S$ are complete. Let $(Q, \m, k)$ be a regular local ring such that $R \cong Q/I$ is a minimal Cohen presentation. By \Cref{l_compatability-Cohen-pres} we may take any minimal Cohen presentation of $S$ which we construct via \Cref{c_lift-weakly-reg}. Thus we have a commutative diagram,
\[
\begin{tikzcd}
Q\ar[rr,"\tilde \varphi"] \ar[dd, two heads, "\pi_R", swap] && Q' \ar[dd, two heads, "\pi_S"] \\ \\
R \ar[rr, "\varphi"] && S
\end{tikzcd}
\]
where $\tilde \varphi$ is weakly regular, $(Q', \m',k')$ is a complete regular local ring, and $S \cong Q' \otimes_Q R$. Since the map $Q\twoheadrightarrow R$ is a minimal Cohen factorization, so is $Q' \to S$. \newline

Let $J = IQ'$, so $S \cong Q'/J$. Let $M$ be an object in $D^f(S/R)$. Let $f$ be a minimal generator of $I$. We will first assume that $\tilde \varphi(f)$ is a minimal generator of $J$. If $\Ext_{Q'/\tilde \varphi(f)}(M,k')$ is bounded, then the projective dimension of $M$ over $Q'/(\tilde \varphi(f))$ is finite \cite[17.3.30]{Christensen/Foxby/Holm:2024} (see  \textit{loc.cit.\!} for a precise definition of projective dimension for complexes). Thus, since $Q/(f) \to Q'/(\tilde \varphi(f))$ is flat, by \cite[8.3.19]{Christensen/Foxby/Holm:2024}, the projective dimension of $M$ over $Q/(f)$ is finite, and thus $\Ext_{Q/(f)}(M,k)$ is bounded. Now, suppose that $\Ext_{Q/(f)}(M,k)$ is bounded. Note that $Q/(f) \to Q'/(\tilde \varphi(f))$ is weakly regular, so by \cite[2.7]{Avramov/Foxby/Halperin:1985} we have that $\Tor^{Q'/\tilde \varphi(f)}(M,k')$ is bounded. Note the theorem cited is only stated for modules, but the proof follows \textit{mutatis mutandis} for complexes. Thus, $\Ext_{Q'/(\tilde \varphi(f))}(M,k')$ is bounded by \cite[5.3.3]{Pollitz:2021} and \cite[6.1]{Felix/Halperin/Thomas:2001}. Hence, in the case where $f$ and $\tilde \varphi(f)$ are minimal generators of $I$ and $J$ respectively, $\Ext_{Q/(f)}(M,k)$ is unbounded if and only if $\Ext_{Q'/(\tilde\varphi (f))}(M,k')$ is unbounded. Thus,
\[
\overline{f} \in \cV_R(M) \iff \cV_{\varphi}(\overline f)\in\cV_S(M).
\]

Now, suppose that $\tilde \varphi(f)$ is not a minimal generator of $J$. We need to show $\overline{f} \in \cV_R(M)$. By \cite[2.2]{Pollitz/Sega:2025} and \cite[6.1]{Felix/Halperin/Thomas:2001}, since $\tilde\varphi(f) \in \m' J \subseteq \m' \ \ann_{Q'}(M)$, we have that $\Ext_{Q'/(\tilde \varphi(f))}(M,k')$ is unbounded. Thus, by the argument above, $\Ext_{Q/(f)}(M,k)$ is unbounded, so $\overline{f} \in \cV_R(M)$.

Lastly, when $f$ is not a minimal generator of $I$, $\overline f = 0$ and thus $\cV_{\varphi}(\overline f) =0$ as well.
\end{proof}

\begin{theorem}\label{t_main-theorem}
Let $\varphi\!:\! \!R \to S$ be a local homomorphism between local rings with algebraically closed residue fields. Let $M$ be an object of $D^f(S/R)$. Then, 
\[
\cV_R(M) = \cV_{\varphi}^{-1}(\cV_S(M))
\]
at the level of affine varieties.
\end{theorem}

\begin{proof}
By the definition of $\cV_R(M)$ and $\cV_S(M)$, we may assume that $R$ and $S$ are complete. Then $\varphi\!:\!R \to S$ factors as
\[
\begin{tikzcd}
&T \ar[dr, two heads,  "\varphi'"] & \\
R\ar[ur, "\dot \varphi"] \ar[rr, "\varphi"] && S
\end{tikzcd}
\]
where $\dot \varphi\!:\!R \to T$ is weakly regular and $\varphi'\!:\!T \to S$ is surjective. Since $\cV_{\varphi} = \cV_{\varphi'}\circ\cV_{\dot \varphi}$, the previous two lemmas yield the desired result.
\end{proof}

\begin{chunk}
We end this section by looking at examples and corollaries of \Cref{t_main-theorem}.
\end{chunk}

\begin{example}
Let $Q = k\llbracket x,y,z,w \rrbracket$, $R=Q/(x^2,w^2)$, and $S=Q/(x^2,xy,yz,zw,w^2)$ with $k$ an algebraically closed field. From \cite[1.4.8]{Briggs/Grifo/Pollitz:2024}, $\cV_S(S) = \mathbf V(\chi_1\chi_5)$ where $k[\chi_1,\dotsc ,\chi_5]$ is the ring of cohomological operators associated to $S$. Let $\varphi\!:\!R \to S$ be the canonical surjection. Then we can identify $\cV_{\varphi}$ with the following matrix,
\[
\cV_{\varphi} = \begin{pmatrix}
1 & 0 \\
0 & 0 \\ 
0 & 0 \\
0 & 0 \\
0 & 1 \\
\end{pmatrix}.
\]
By \Cref{t_main-theorem},
\begin{align*}
\cV_R(S) &=  \cV_{\varphi}^{-1}(\cV_S(S)) \\
&= \{(a_1,a_2)\in \mathbb A^2_k \mid (a_1,0,0,0,a_2) \in \mathbf V(\chi_1\chi_5)\} \\
&= \{(a_1,a_2)\in \mathbb A^2_k \mid a_1=0 \text{ or } a_2=0\} \\
&= \mathbf V(\chi_1'\chi_2').
\end{align*}
Where $k[\chi_1', \chi_2']$ is the ring of cohomological operators associated with $R$.
\end{example}

\begin{remark}\label{r_bergh-jorgensen}
Let $(Q,\m,k)$ be a regular local ring with algebraically closed residue field, and $I\subseteq J$ ideals in $Q$ contained in $\m^2$ such that they are minimally generated by regular sequences.  Let $R = Q/I$, $S=Q/J$, and $\varphi \colon R \to S$ be the canonical surjection. If we assume that $\cV_{\varphi}\!:\!I/\m I \to J / \m J$ is injective, then by \Cref{t_main-theorem}, for any pair of finitely generated $S$ modules, $M$ and $N$, the following equality holds
\[
\cV_{\varphi}(\cV_{R}(M,N)) = \cV_{S}(M,N) \cap W
\]
where $W = \text{Im}(\cV_{\varphi})$. This was first proved in \cite[3.1]{Bergh/Joregensen:2015}.
\end{remark}

\begin{corollary}\label{co_ker}
Let $\varphi\!:\!R \to S$ be a local homomorphism between local rings with algebraically closed residue fields. Suppose that $M$ is an element of $D^f(S/R)$. If $\cV_{S}(M) = \{0\}$, then $\cV_{R}(M) = \ker(\cV_{\varphi})$.
\end{corollary}

\begin{proof}
By \Cref{t_main-theorem}, $\cV_R(M) \!= \!  \cV_{\varphi}^{-1}(\cV_S(M))$. Since $\cV_S(M)$ is assumed to be zero, we have that $\cV_R(M)$ is the preimage of $\{0\}$ under $\cV_{\varphi}$, otherwise known as the kernel of $\cV_{\varphi}$. 
\end{proof}

\begin{chunk}
The previous corollary recovers a result from \cite[3.6]{Briggs/Grifo/Pollitz:2022} which states the result for a more restrictive class of morphisms and only for $M = S$. Note that in general, it is unclear when an object in $D^f(S)$ has $\{0\}$ as its cohomological support variety (over $S$). However, by \cite{Pollitz:2021}, we must have that $S$ is a complete intersection if $\{0\}$ can be realized as the cohomological support variety of a finitely generated module. Moreover, by \cite[2.8]{Briggs/Grifo/Pollitz:2024} we know that if $M \in D^f(S)$ is such that $V_S(M) = \{0\}$ then $S$ can not be Cohen-Macaulay unless it is a complete intersection.
\end{chunk}

\begin{corollary}\label{co_full-dim}
Let $\varphi\!:\!R \to S$ be a local homomorphism between local rings with algebraically closed residue fields. Let $(Q,\m,k)$ and $(Q',\m',k')$ be regular local rings such that $\widehat R \cong Q/I$ and $\widehat S \cong Q'/J$ are minimal Cohen presentations. Let $\tilde \varphi\!:\! Q \to Q'$ be a lift of $\varphi$ and suppose that $\tilde\varphi(I) \subseteq \m' J$. Then for any $M$ an object of $D^f(S/R)$, we have $\cV_R(M) = I / \m I$.
\end{corollary}

\begin{proof}
In this specific scenario, the map $\cV_{\varphi}:I / \m I \to \! J / \m' J$ is the zero map. By \Cref{t_main-theorem}, $\cV_R(M) = \cV_{\varphi}^{-1}(\cV_S(M))$. By definition, $0 \in \cV_S(M)$, so the preimage of $\cV_S(M)$ under $\cV_{\varphi}$ is all of $I/\m I$. 
\end{proof}

\begin{remark}
When $Q = Q'$, this result can also be recovered from a paper of Pollitz and \cb{S}ega \cite[2.2]{Pollitz/Sega:2025}. Let $E$ be the Koszul complex associated to the quotient $Q/I$. The results in their paper relate the Poincar\'{e} series of $M$ over $E$ to the Poincar\'{e} series of $M$ over $Q$. Since the dimension of $\cV_R(M)$ gives the complexity of $M$ over $E$ \cite{Pollitz:2021}, their result implies that the dimension of $\cV_R(M)$ must be maximal. In fact, they show something even stronger. They show that $\Ext_E(M,k)$ is free over the ring of cohomological operators \cite[2.6]{Pollitz/Sega:2025} which implies the statement about the cohomological support varieties.
\end{remark}

\section{Finite Flat Dimension}\label{s_ffd}
In this section, we look at restriction of scalars along maps of finite projective dimension. In particular, we prove a dimension inequality on support varieties along a local map of finite flat dimension. Using this, we recover Avramov's result that the complete intersection property localizes \cite{Avramov:1977}.

\begin{chunk}\label{c_homotopy-lie-alg}
Let $(R,\m_R,k)$ and $(S,\m_S, k')$ be complete local rings, and $\varphi\!:\!R \to S$ a local homomorphism of finite flat dimension. Consider a Cohen factorization
\[
\begin{tikzcd}
&T\ar[dr, two heads, "\varphi'"]&\\
R\ar[ur, "\dot \varphi"] \ar[rr, "\varphi"]&&S .
\end{tikzcd}
\]
Since $\dot \varphi$ is weakly regular and $\varphi$ has finite flat dimension, by \cite[3.5]{Avramov/Foxby/Halperin:1985} we have
\[
\pi_2(R) \otimes_kk' \cong \pi_2(T) \hookrightarrow\pi_2(S),
\]
where $\pi_*(-)$ denotes the homotopy lie coalgebra. A deep understanding of $\pi_*(-)$ is not needed for this paper, but for the interested reader we refer them to \cite[5.4]{Avramov/Buchweitz:2000b}, \cite[Chapter 10]{Avramov:2010}, and \cite[1.7]{Briggs/Grifo/Pollitz:2024} for more information on the homotopy lie (co)algebra and its interaction with cohomological support varieties. The two pieces of information regarding the homotopy lie coalgebra that are relevant to this paper are
\begin{enumerate}
\item $\pi_2(R) \cong I/\m_Q I$ as $k$-vector spaces, where $Q/I \cong R$ is a minimal Cohen presentation. Thus, we can view $\cV_R(M)$ inside of $\pi_2(R)$ for any local ring $R$ and any $M$ in $D^f(R)$.
\item The map $\pi_2(T)\hookrightarrow \pi_2(S)$ is the map $\cV_{ \varphi'}$ when identifying $\pi_2(T)$ and $\pi_2(S)$ with the source and target of $\cV_{\varphi '}$ respectively. In particular, $\cV_{\varphi '}$ is injective.
\end{enumerate}
Note that these facts can also be used to prove the naturality of the map $\cV_{\varphi}$ with respect to the choice of Cohen presentations of $R$ and $S$ as in \Cref{l_compatability-Cohen-pres}.
\end{chunk}

\begin{theorem}\label{t_finite-flat-dim}
Let $(R,\m_R,k)$ and $(S,\m_S,k')$ be local rings and $\varphi\! : \! R \to S$ a local homomorphism of finite flat dimension. Then
\[
\dim \cV_R(R) \les \dim \cV_S(S).
\]
\end{theorem}

\begin{proof}
We first note that we can reduce to the case where both $R$ and $S$ are complete. Note that $R \to \widehat {S}$ is a local map of finite flat dimension as $S \to \widehat S$ is flat. Thus $\widehat R \to \widehat S$ has finite flat dimension by \cite[2.7]{Avramov/Foxby/Halperin:1985}. Since $\cV_R(M) = \cV_{\widehat{R}}(\widehat M)$ for all $M \in D^f(R)$, we can take $R$ and $S$ to be complete. \newline

Consider the setup in \Cref{c_homotopy-lie-alg}. By \cite[App., Theor\'eme 1]{Bourbaki:1983}, there exist residual algebraic closures, $\tilde S$ and $\tilde T$ of $S$ and $T$ respectively such that the following diagram commutes:
\[
\begin{tikzcd}
&&\tilde T\ar[dr, two heads] && \\
&T\ar[ur] \ar[dr, two heads]&&\tilde S\\
R\ar[rr] \ar[ur]&&S \ar[ur]& \ \ .&
\end{tikzcd}
\]
Since $R \to T$ and $T \to \tilde T$ are weakly regular so is the map $R \to \tilde T$ \cite[4.4]{Avramov/Foxby:1998}. Moreover, $R \to \tilde S$ is of finite flat dimension since $R \to S$ is of finite flat dimension and $S \to \tilde {S}$ is flat. By \Cref{r_dim-equality}, $\dim \cV_{\tilde S}(\tilde{S}) = \dim \cV_S(S)$, so we may assume we have the following Cohen factorization of $\varphi$,
\[
\begin{tikzcd}
&T\ar[dr, two heads, "\varphi'"]& \\
R \ar[ur, "\dot \varphi"] \ar[rr, "\varphi"]&& S 
\end{tikzcd}
\]
where $T$ and $S$ have the same algebraically closed residue field. 

By \Cref{l_restriction-surjections}, $\cV_T(S) = \cV_{\varphi'}^{-1}(\cV_S(S))$. Thus, $\dim \cV_T(S) \les \dim \cV_S(S)$ since $\cV_{\varphi'}$ is an injective, linear map. Now, $S$ has finite flat dimension over $R$, and thus has finite flat dimension over $T$ \cite[2.7]{Avramov/Foxby/Halperin:1985}. Since $S$ is a finitely generated $T$-module, this shows that $S$ has finite projective dimension over $T$ as $T$ is noetherian. Hence, $\cV_T(S) = \cV_T(T)$ by \cite[3.3.2]{Pollitz:2019} and \cite[2.3]{Briggs/Grifo/Pollitz:2022}. If $k$ is the residue field of $R$ and $k'$ is the residue field of $T$, by \Cref{r_fiber-product} the following isomorphism holds, $\cV_T(T) \cong \cV_R(R)\times_{\spec k} \spec k'$, and thus $\dim \cV_R(R) = \dim \cV_T(T)$ by \Cref{r_dim-equality}. Putting this all together,
\[
\dim \cV_R(R) = \dim \cV_T(T) = \dim \cV_T(S) \les \dim \cV_S(S). \qedhere
\]
\end{proof}

\begin{remark}
In light of \cite[3.3.2 (3)]{Pollitz:2019}, one may think of the dimension of the cohomological support variety of a ring as a measure of how singular the ring is. Thus, the previous theorem falls in line with other results that suggest that the singularity of a ring can only get worse along a map of finite flat dimension.
\end{remark}

\begin{corollary}\label{co_localization-ci}
Let $(R,\m,k)$ be a local ring. If $R$ is a complete intersection, then for every prime ideal $\p \in \Spec R$, the localization $R_\p$ is a complete intersection.
\end{corollary}

\begin{proof}
Consider the morphism $R_\p \to \widehat{R}_\q$ where $\q$ is some prime ideal in $\widehat R$ lying over $\p$. By \Cref{t_finite-flat-dim}, $\dim \cV_{R_\p}(R_\p) \les \dim \cV_{\widehat {R}_\q}(\widehat{R}_\q)$. By \Cref{c_dim-equality-cohen} and \cite[5.2.9]{Pollitz:2021}, we have the inequality $\dim \cV_{\widehat{R}_\q}(\widehat R_\q) \les \dim \cV_{\widehat R} \widehat R$. Indeed, if $Q \to \widehat{R}$ is a Cohen presentation, so is $Q_{\mathfrak r} \to \widehat R_  \q$ where $\mathfrak r$ is some prime lying over $\q$. Thus, if $E = \Kos^Q(f_1,\dotsc, f_n)$ is a Koszul complex corresponding to the quotient $Q \to \widehat R$, then $E_{\mathfrak r} := E \otimes _Q Q_\mathfrak r$ is a Koszul complex corresponding to the quotient $Q_{\mathfrak r} \to \widehat R_\q$. Hence, since $Q \to Q_{\mathfrak r}$ is flat, any semi-free $E$-resolution of $\widehat R$ will yield a semi-free $E_{\mathfrak r}$ resolution of $\widehat R_\q$ once tensored with $Q_{\mathfrak r}$. Thus the complexity (see \cite[5.2.8]{Pollitz:2021}) of $\widehat R$ over $E$ is at least that of $\widehat R_\q$ over $E_{\mathfrak r}$. Hence, $\dim \cV_{\widehat R_\p}(\widehat R_\p) \les \dim \cV_{\widehat R}(\widehat R)$, so 
\[
\dim \cV_{R_\p}(R_\p) \les \dim \cV_{\widehat {R}_\q}(\widehat{R}_\q) \les \dim \cV_{\widehat R}(\widehat R) = \dim \cV_R(R).
\]
By \cite[3.3.2]{Pollitz:2019} we are done.
\end{proof}

\begin{corollary}
Let $R \to S$ be a local map of finite flat dimension. If $S$ is a complete intersection, then $R$ is a complete intersection as well. 
\end{corollary}
\begin{proof}
This is immediate by \Cref{t_finite-flat-dim} and \cite[3.3.2]{Pollitz:2019}. 
\end{proof}

\begin{chunk}
The previous corollary was also proved by Avramov using Andr\'{e}-Quillen homology \cite[5.10]{Avramov:1999}.
\end{chunk}

\bibliographystyle{alpha}
\bibliography{references}

\end{document}